\documentclass[12pt]{amsart}
\usepackage{amsmath, amssymb, amscd, mathrsfs, slashed, enumerate, url}
\usepackage{color}
\usepackage{xcolor}
\usepackage[all,cmtip]{xy}
\usepackage{multicol}
\usepackage{indentfirst}
\usepackage{latexsym}
\usepackage{bm}
\usepackage{graphicx}
\usepackage{subfigure,esint}
\usepackage{float}
\usepackage{verbatim}
\usepackage{comment}
\usepackage[top=1in, bottom=1in, left=1.25in, right=1.25in]{geometry}
\usepackage{epsfig,dsfont,amsthm,amsfonts,amsbsy}
\usepackage[colorlinks]{hyperref}
\usepackage[toc,page]{appendix}
\usepackage{tikz-cd}
\usepackage{tikz}
\usetikzlibrary{matrix}

\newcommand{\lan}{\langle }
\newcommand{\ran}{\rangle}

\newcommand{\MM}{\mathcal{M}}
\newcommand{\MO}{\mathcal{O}}

\newtheorem{theorem}{Theorem}[section]

\newtheorem{corollary}[theorem]{Corollary}

\newtheorem{definition}[theorem]{Definition}

\newtheorem{example}[theorem]{Example}

\newtheorem{lemma}[theorem]{Lemma}

\newtheorem{proposition}[theorem]{Proposition}

\newcommand{\MC}{\mathcal{C}}
\newcommand{\Tr}{\mathrm{Tr}}

\newcommand{\st}{\star}
\newcommand{\we}{\wedge}
\newcommand{\pa}{\partial}

\newcommand{\ti}{\times}

\newcommand{\al}{\alpha}

\newcommand{\na}{\nabla}
\newcommand{\ep}{\epsilon}

\newcommand{\MF}{\mathcal{F}}

\newcommand{\MI}{\mathcal{I}}
\newcommand{\ZT}{\mathbb{Z}_2}

\newcommand{\Id}{\mathrm{Id}}

\newcommand{\SLC}{SL(2,\mathbb{C})}

\newcommand{\tM}{\tilde{M}}

\newcommand{\MD}{\mathcal{D}}

\newcommand{\divv}{\mathrm{div}}

\newcommand{\tL}{\tilde{L}}

\newcommand{\tv}{\tilde{v}}

\newcommand{\supp}{\mathrm{supp}(T)}

\newcommand{\ZH}{\mathbb{Z}_3}
\newcommand{\ttau}{\tilde{\tau}}

\usepackage[utf8]{inputenc}
\usepackage[english]{babel}
\usepackage{fancyhdr}
\usepackage{accents}

\begin{document}
	\title[Existence of nondegenerate $\ZT$ harmonic 1-forms via $\mathbb{Z}_3$ Symmetry]{Existence of nondegenerate $\ZT$ harmonic 1-forms via $\mathbb{Z}_3$ Symmetry}
	\author{Siqi He} 
	\date{\today}
	\address{Simons Center for Geometry and Physics, StonyBrook University\\Stonybrook, NY 11794}
	\email{she@scgp.stonybrook.edu}
	
		\maketitle
	
 \centerline{{\it Dedicated to Professor Rafe Mazzeo on the occasion of his 60th birthday}} 
	
	\begin{abstract}
		Using $\ZH$ symmetry, we present a topological condition for the existence of the $\ZT$ harmonic 1-forms over Riemannian manifold. As a corollary, if $L$ is an oriented link on $S^3$ with determinant zero, then there exists a nondegenerate $\ZT$ harmonic 1-form over the 3-cyclic branched covering of $L$. Furthermore, we found infinite number of rational homology 3-spheres that admit a nondegenerate $\ZT$ harmonic 1-form.
	\end{abstract}
	\section{Introduction}
	\subsection{The multivalued harmonic 1-forms}
	Let $M$ indicates a compact smooth oriented n-dimensional manifold with Riemannian metric $g$. A multivalued harmonic 1-form on $M$ is defined by a data set $(L,\MI,v)$, where $L$ is a smooth oriented codimension 2 embedded submanifold of $M$, $\MI$ is a flat real line bundle over $M$ with monodromy $-1$ along any small loop linking $L$, and $v\in L^2$ is a $\MI$ valued 1-form defined on $M\setminus L$ that satisfies 
	$$dv=d\st_gv=0,
	$$
	where $\st_g$ is the Hodge star operator using Riemannian metric $g$ and the differential $d$ make sense using the flat structure on $\MI$. 
	
	A multivalued harmonic form $v$ is called a $\ZT$ harmonic 1-form, named in \cite{taubes2013psl2c}, if $v\neq 0$ and bounded. The $\ZT$ harmonic 1-forms play an essential role in gauge theory with non-compact gauge group, which characterizes non-compactness behavior of the $\SLC$ flat connection equations, according to the pioneering work of Taubes \cite{taubes2013compactness,taubes2013psl2c}. We refer to \cite{donaldson2019deformations,takahashi2015moduli,taubeswu2020,taubes2014zero,zhang2017rectifiability} as works that investigate the analytic aspect of the $\ZT$ harmonic 1-form.
	
	Therefore, constructing examples of $\ZT$ harmonic forms is of great interest. Currently, most known examples of $\ZT$ harmonic 1-forms over closed manifolds are coming from K\"ahler geometry. The main goal of our paper is to construct examples of $\ZT$ harmonic 1-forms over real closed manifolds.
	\subsection{The expansions of the multivalued harmonic 1-forms}
	A local model of $\ZT$ harmonic 1-form will be the following: let $M=\mathbb{C}$, $L=\{0\}$, we define $v_k=z^{k-\frac12}dz$ for $k$ is a non-negative integer. As $\sqrt{z}$ is two valued over $\mathbb{C}\setminus \{0\}$, we could regard $v_k$ as either a section of $\MI\otimes T^{\st}M|_{M\setminus L}$, where $\MI$ is the M\"obius bundle or a multivalued 1-form over $M\setminus L$. Over any compact set on $\mathbb{C}$, $v_k$ lies in $L^2$ and when $k\geq 1$, $v_k$ is bounded.
	
	In general, a multivalued harmonic 1-form will not necessarily be bounded near $L$. Moreover, not every Riemannian manifold admits bounded $\ZT$ harmonic 1-form. If the Riemannian manifold has non-negative Ricci curvature, then there doesn't exist $\ZT$ harmonic 1-form, according to the Weitzenb\"ock identity.
	
	The following descriptions of multivalued harmonic 1-forms are based on \cite{donaldson2019deformations}, see also \cite{haydys2020seiberg}. Given a multivalued harmonic 1-form $(L,\MI,v)$, the flat bundle $\MI$ defines a map $\pi_1(M\setminus L)\to \mathbb{Z}_2$, which we could use the kernel to define a double branched covering $p:\tM\to M$, branched along $L$, together with $\sigma:\tM\to \tM$ an involution. In addition, $p^{\st}v$ will be a well-defined 1-form on $\tM$ satisfying the harmonic equations for $p^{\st}g:$
	$$
	d(p^{\st}v)=d(\st_{p^{\st}g}p^{\st}v)=0.
	$$
	Despite the fact that $p^{\st}g$ is not smooth along $L$, $p^{\st}g$ is Lipschitz, and the Hodge theorem holds for this type of metric \cite{teleman1983theindex,mazzeohunsicker05}. As a result, $p^{\st}v$ represents a homology class $\chi\in H^1_-(\tM;\mathbb{R})$, where $H^1_-(\tM;\mathbb{R})$ is the $-1$ eigenspace of the induced map of the involution $\sigma^{\st}:H^1(\tM;\mathbb{R})\to H^1(\tM;\mathbb{R})$. 
	
	The converse of the preceding discussion is equally valid. Given $L$ a embedded codimension 2 submanifold of $M$ with flat $\mathbb{R}$ bundle $\MI$ over $M\setminus L$ satisfying the above monodromy assumption. We write the corresponding double branched covering $p:\tM\to M$. For any $\chi\in H^1_-(\tM;\mathbb{R})$, there exists a multivalued harmonic 1-form $v_{\chi}$ over $M$ such that $[p^{\st}v_{\chi}]=\chi$. As a result, the first Betti number of the branched covering $\tM$ entirely determines the multivalued harmonic1-form.
	
	Let $N$ be the normal bundle of $L$, then a neighborhood of $L$ could be identified with a neighborhood of the zero section of $N$. Using the coorientation of $L$, we could regard $N$ as a complex vector bundle. Near a neighborhood $U$ of $l\in L$, we could choose a trivialization of $N|_U$ and find a complex coordinate $z=re^{i\theta}$ on $N|_U$ such that $(z,t)$ is a coordinate on $U$ with $L\cap U=\{z=0\}$, and $t$ is a coordinate on $L$. In addition, $z^{\frac12+k}$ could be understood as a multivalued section of $N^{\frac12+k}$.
	
	By the work of Donaldson \cite{donaldson2019deformations}, over $U$, there exists an expansion
	\begin{equation}
		\label{eq_expanison}
		v\sim \Re(d(A(t)z^{\frac12})+d(B(t)z^{\frac32}))+\MO(r^{\frac32-\ep}).
	\end{equation}
	There are two ways to interpret the $A(t)$, $B(t)$. After we choose a trivialization of $N$ over $U$, we could regard $A,B$ are complex functions over $L\cap U$. In particular, if $N$ is trivial, then we could regard $A,B$ are complex functions over $L$, which are the cases that is most considered in this paper. There is also a global version of the expansion explained in \cite[Section 3]{donaldson2019deformations} without assuming $N$ is trivial, where $A\in \Gamma(N^{-\frac12})$ and $B\in \Gamma(N^{-\frac32})$ could be understood as sections of bundles, while $A(t)z^{\frac12}$ and $B(t)z^{\frac32}$ are pairing between suitable bundles. 
	
	From the expansion \eqref{eq_expanison}, a multivalued harmonic 1-form is bounded if the $A\equiv 0$. Now, we will introduce an important class of $\ZT$ harmonic 1-forms, which could be regarded as "generic" $\ZT$ harmonic 1-forms.
	\begin{definition}
		A $\ZT$ harmonic 1-form $(L,\MI,v)$ is called nondegenerate if the "B" terms of the expansion of $v$ in \eqref{eq_expanison} is nowhere vanishing along $L$. 
	\end{definition}
	
	The following example explains the reason to introduce the above definition. Suppose $M$ is a Riemannian surface and $(L,\MI,v)$ is a multivalued harmonic form. Then we can find a meromorphic differential $\al$ over $M$ such that $\Re(\al)=v\otimes v$. $v$ is a bounded multivalued harmonic form if and only if $\al$ is a holomorphic quadratic differential. $v$ is nondegenerate if and only if $\al$ is a quadratic differential with simple zeroes. 
	
	In general, given a homology class $\chi\in H^1_-(\tM,\mathbb{R})$, let $v_{\chi}$ be the corresponding multivalued harmonic forms. The "$A$" term for the corresponding $v_{\chi}$ might not vanish. In the rest of the paper, we will introduce a $\ZH$ symmetry to make $A(t)\equiv 0$. The $\ZH$ symmetry have also been used by Taubes and Wu \cite{taubeswu2020} in getting models for $\ZT$ harmonic 1-form with graphic singularity over $\mathbb{R}^3$. We also refer \cite{mazzeoandriyroyosuke2022,takahashi2015moduli} for related analytic aspects.
	\subsection{Main results}
	Now, we will introduce our main results.
	%We also refer \cite{taubes2014zero,zhang2017rectifiability} for study in this direction.
	\begin{definition}
		\label{assumption} We say $(\tM,M,\chi)$ is a $\ZH$ symmetric triple if there exists a $\ZH$ action $\tau:M\to M$ with $\tau^3=\Id$ and fixed locus $\mathrm{Fix}(\tau)=L$, such that
		\begin{itemize}
			\item [(i)] $\tau$ has a lifting $\ttau$ on $\tM$ such that $\ttau\circ\sigma=\sigma\circ\ttau$.
			\item [(ii)] $\chi\in H^1_-(\tM;\mathbb{R})$ is invariant under $\ttau$, $\ttau^{\st}\chi=\chi$.
		\end{itemize}
	\end{definition}
	It's worth noting that the above definition does not necessitate the existence of a Riemannian metric on $M$, which is a fully topological condition. The following is a statement of our main theorem.
	
	\begin{theorem}
		\label{main_theorem}
		Suppose $(\tM,M,\chi)$ is a $\ZH$ symmetric triple, then the following holds:
		\begin{itemize}
			\item [(i)]For any $\ZH$ invariant smooth Riemannian metric on $M$, there exists a bounded multivalued harmonic 1-form $(L,\MI,v)$ with $[p^{\st}v]=\chi$ and $L$ is the fixed point of the $\ZH$ action.
			\item [(ii)]Suppose when $M$ is a 3-manifold, then for generic $\ZH$ invariant smooth Riemannian metric on $M$, $(L,\MI,v)$ is nondegenerate.
		\end{itemize} 
	\end{theorem}	
	
	Over 3-manifold, we construct examples of $\ZH$ symmetric triple using cyclic branched covering of oriented links.
	\begin{corollary}
		Let $M$ be a rational homology 3-sphere with $L$ be an oriented link on $M$ and let $M_k$ be the k-cyclic branched covering of $L$. Suppose the determinant of $L$ vanishes, then there exists a nondegenerate $\ZT$ harmonic 1-form on $M_3$.
	\end{corollary}
	
	By \cite{mclean98defomation}, harmonic 1-form will generate smooth deformations of special Lagrangian submanifolds, whereas nondegenerate $\ZT$ harmonic 1-form is responsible for the branched deformations \cite{he2022brancehd}. The existence of a multivalued harmonic 1-form over a manifold with vanishing first Betti number will be of particular importance. This means that the special Lagrangian submanifold will exhibit branched deformation rather than smooth deformation. In particular, we construct examples of multivalued harmonic 1-forms over rational homology 3-spheres.
	\begin{theorem}
		\label{thm_infinity}
		There exists infinite number of rational homology 3-spheres which admit nondegenerate $\ZT$ harmonic 1-form.
	\end{theorem}
	\textbf{Acknowledgements.} The author is grateful to Simon Donaldson for introducing this problem and providing countless helpful discussions. The author would also like to thank Rafe Mazzeo, Donghao Wang, and Langte Ma for their valuable input.
	
	\section{Existence and transversality}
	Now, we will prove Theorem \ref{main_theorem}. We will first use the $\ZH$ symmetry to construct bounded multivalued 1-form, and then establish a transversality result for generic $\ZH$ metric.
	\subsection{Existence of $\ZT$ harmonic 1-forms.}
	Proof of Theorem \ref{main_theorem} (i): Let $(\tM,M,\chi)$ be a $\ZH$ symmetric triple and $p:\tM\to M$ be the double branched covering induced by flat bundle $\MI$ and $\sigma:\tM\to\tM$ be the involution. Let $g$ be any $\ZH$ invariant smooth Riemannian metric over $M$, then $p^{\st}g$ is a Lipschitz metric over $\tM$. By the Hodge theorem for Lipschitz metric \cite{teleman1983theindex,wang93modulispaces}, for any $\chi\in H^1_-(\tM;\mathbb{R})$, there exists $\tv$, a harmonic representative of $\chi$, such that $$d\tv=d\st_{p^{\st}g}\tv=0.$$ 
	
	Moreover, as $\sigma^{\st}\chi=-\chi$ and $\sigma^{\st}p^{\st}g=p^{\st}g$, by the uniqueness of Hodge representative, we have $\sigma^{\st}\tv=-\tv$. As $p^{\st}g$ is also $\ZH$ invariant, for the $\ZH$ action $\ttau$ on $\tM$, we have $$d\ttau^{\st}\tv=d\st_{p^{\st}g}\ttau^{\st}\tv=0.$$
	
	As $\ttau^{\st}\chi=\chi\in H^1(\tM;\mathbb{R})$, by the uniqueness of Hodge representative, we have $\ttau^{\st}\tv=\tv$. As $\sigma^{\st}\tv=-\tv$, $\tv$ induces a multivalued harmonic 1-form $v$ on $M$ with $p^{\st}v=\tv$. Moreover, as $\ttau^{\st}\tv=\tv$ and $\ttau\circ\sigma=\sigma\circ\ttau$, we have $\tau^{\st}v=v$.
	
	Now we consider the local expansion of $v$. Let $x\in L$, $v\in N_x$, as $\tau(x)=x$, $\tau$ induces an action $\tau^{\st}:N_x\to N_x$ as $\tau^{\st}(v)=\frac{d}{dt}|_{t=0}\tau\circ\gamma(t)$, where $\gamma(t)$ is any curve with $\gamma(0)=x$ and $\gamma'(0)=v$. As $\tau$ preserve the Riemannian metric and $L$ is the fixed point of $\sigma$, the image of $\tau^{\st}v$ lies in $N_x$. As $\tau^{3}=\Id$ and $N_x$ is a complex vector space with "$z$" a section, $\tau^{\st}$ is a order 3 representation on a complex vector space, which must be $\tau^{\st}z=e^{\frac{2\pi i}{3}}z$. 
	
	We write the expansion of $v$ as $$v\sim \Re(d(A(t)z^{\frac12})+d(B(t)z^{\frac32}))+\MO(r^{\frac32-\ep}).$$ As we have $$\tau^{\st}v=v,\;\tau^{\st}z^{-\frac12}dz=e^{\frac{\pi i}{3}}z^{-\frac12}dz\neq z^{-\frac12}dz,$$ we obtain $A\equiv 0$. Note that $\tau^{\st}z^{\frac12}dz=-z^{\frac12}dz$, while $z^{\frac12}dz$ and $-z^{\frac12}dz$ are the same as multivalued forms, so the above symmetry doesn't cancel the remaining terms.
	\subsection{Perturbation to nondegenerate $\ZT$ harmonic 1-form.}
	Given a $\ZH$ symmetric triple $(\tM,M,\chi)$, let $k$ be a large integer and let $\MM$ be the space of $\ZH$ invariant $\MC^{k}$ Riemannian metric on $M$, which will be a Banach space. By Theorem \ref{main_theorem} (i), for any $g\in\MM$, we could obtain a bounded multivalued harmonic 1-form $(L,\MI,v)$ such that $[p^{\st}v]=\chi$. As the initial condition is $\ZH$ invariant, the constructions in the rest of this subsection are all $\ZH$ invariant. 
	
	We assume that the normal bundle $N$ of $L$ is trivial, such that we could write the expansion of $v$ as
	\begin{equation}
		v\sim \Re(B(t)z^{\frac12}dz)+\MO(r)
	\end{equation}
	where $B$ is a complex function over $L$.
	
	We now try to understand the variation of $B$ with respect to $\ZH$ invariant Riemannian metric. We fixed a based metric $g_0\in \MM$ and choose any 1-parameter family of $g_s\in \MM$. We could apply Theorem \ref{main_theorem} (i) for the metric $g_s$, we could find a family of bounded multivalued 1-form $v_s$ such that $d^{\st_{g_s}}v_s=0$. 
	
	We write $\dot{v}_0=\frac{d}{ds}|_{s=0}v_s$ and $\dot{\MD}=\frac{d}{ds}|_{s=0}d^{\st_s}$, then taking derivative of $s$ in $d^{\st_{g_s}}v_s=0$, we obtain
	\begin{equation}
		\label{eq_variation}
		\dot{\MD} v_0+d^{\st_{g_0}}\dot{v}_0=0.
	\end{equation}
	
	We need a Green function description of "B" terms introduced in \cite[Section 2, 3]{donaldson2019deformations}. Using the flat structure on $\MI$, we could define the multivalued Laplacian operator $\Delta_{g_0}$, with the Green function of $\Delta_{g_0}$ as $G(p,p')$, where $p,p'\in M$. Over \cite{donaldson12kahlermetric}, there is an explicit expressions of $G$ in terms of Bessel functions in the model case. The Green function $G(p,p')$ is smooth on the complement of the diagonal and has a standard singularity on the diagonal as Newton kernel. 
	
	Fixed $l\in L$ and in a neighborhood $l$ and let $t$ be the coordinates on $L$ centered at $l$. We write $p=(z,t)$ and $p'=(z',t')$, then for $p'\neq p$, we have the expansion along $\{z=0\}$
	\begin{equation}
		G(p,p')\sim \sum_{k,\nu\geq 0} a_{k,\nu}(t,p')e^{i(\nu+\frac12)\theta}r^{\nu+\frac12+2k}.
	\end{equation}
	We write $G_1(t,p')=a_{1,0}(t,p')$, which is singular only when $z'=0,\;t'=t$, is a multivalued function on $p'$. In this neighborhood, for the family of Riemannian metric $g_s$ with corresponding multivalued harmonic 1-form $v_s$, we write the "B" terms of $v_s$ as $B_s(t)$, which is a complex function over $L$.
	
	For \eqref{eq_variation}, by \cite[Section 2]{donaldson2019deformations}, we could write
	\begin{equation}
		\frac{d}{ds}|_{s=0}B_s(t)=\int_MG_1(t,p')\dot{\MD} v_0 dp'\in \mathbb{C},
	\end{equation}
	where $dp'$ is the volume form of $g_0$. As $G_1$ and $\dot{\MD} v_0$ are both multivalued functions, $G_1\dot{\MD} v_0$ is a complex function.
	
	Let $T:=\frac{d}{ds}|_{s=0}g_s$ is a symmetric $(0,2)$ tensor and the divergence $\divv T$ is a $(0,1)$ tensor which could be defined as $$(\divv T)V=\Tr(W\to (\na_{W}T)(V)),$$
	where $V$ and $W$ are vectors on $M$.
	
	A straight forward computation shows
	\begin{equation}
		\dot{\MD} v_0=\frac12 \lan d\Tr(T),v_0\ran-\lan \mathrm{div}\;T, v_0\ran-\lan T,\na v_0\ran,
	\end{equation}
	where the inner product is taken using $g_0$.
	
	Therefore, we obtain 
	\begin{equation}
		\label{eq_Bderivative}
		\frac{d}{ds}|_{s=0}B_s(t)=\int_M G_1 \dot{\MD} v_0 dp'=\int_M G_1(\frac12 \lan d\Tr(T),v_0\ran-\lan \mathrm{div}\;T, v_0\ran-\lan T,\na v_0\ran) dp'.
	\end{equation}
	
	\begin{lemma}
		\label{lem_formula}
		Let $X_0$ be the dual vector of $v_0$ under the isomorphism induced by the Riemannian metric $g_0$, then we obtain the following equalities.
		\begin{itemize}
			\item [(i)]
			$G_1\divv(T)(X_0)=\divv(T(G_1X_0))-G_1\lan T,\na v_0 \ran-T\lan dG_1,v_0\ran.$
			\item [(ii)] $G_1\lan d\Tr(T),v_0\ran dp'=d(G_1\Tr(T)\we \st v_0)-\Tr(T)\lan dG_1,v_0\ran dp'.$
		\end{itemize}
	\end{lemma}
	\begin{proof}
		For (i), we compute in the orthonormal frame. Let $p\in M$, we write $E_1,\cdots,E_n\in T_pM$ be a orthonormal frame with dual frame $\theta_1,\cdots,\theta_n$. Then we write 
		$$
		T=\sum_{i,j=1}^nT_{ij}\theta^i\otimes\theta^j,\;v_0=\sum_{i=1}^n\al_i\theta^i,\;X_{0}=\sum_{i=1}^n\al_i E_i,
		$$
		then $\divv(T)=\sum_{i,j=1}^n(E_iT_{ij})\theta^j$. 
		We compute
		\begin{equation}
			\begin{split}
				\divv(T(G_1X_0))&=\sum_{i,j=1}^n\divv(T_{ij}\theta^iv_0 G_1)=\sum_{i,j=1}^nE_i(T_{ij}\al_jG_1)\\
				&=\sum_{i,j=1}^n(E_iT_{ij})\al_j G_1+ T_{ij}(E_i\al_j)G_1+T_{ij}\al_j(E_iG_1)\\
				&=G_1(\divv(T)X_0)+G_1\lan T,\na v_0\ran+T\lan dG_1,v_0\ran.
			\end{split}
		\end{equation}
		For (ii), let $\st$ be the Hodge star operator of $g_0$, then we compute
		\begin{equation}
			\begin{split}
				G_1\lan d\Tr(T),v_0\ran dp'=G_1 d\Tr(T)\we \star v_0=d(G_1\Tr(T)\we \st v_0)-\Tr(T)dG_1\we\st v_0.
			\end{split}
		\end{equation}
	\end{proof}
	
	Let $U_{l}$ be any small neighborhood of $t\in L \subset M$, then, for $p'\in U_{l}^{c}$, $G_1(t,p')$ is a smooth function. We denote $\supp$ the support of $T$ and suppose $\supp$ lies on $U_l^{c}$, then we don't need to worry about the boundary terms coming from the singularity of $G_1$ during the integral by part processes. 
	
	By Lemma \ref{lem_formula}, \eqref{eq_Bderivative} could be rewritten as
	\begin{equation}
		\frac{d}{ds}|_{s=0}B_s(t)=\int_MG_1 \dot{\MD} v_0dp'=\int_M (T(\na G_1,v_0)-\Tr(T)\lan dG_1,v_0 \ran)dp'.
	\end{equation}
	If we define $S:=\frac12(dG_1\otimes v_0+v_0\otimes dG_1)$, which is a symmetric (0,2) tensor. In addition, as $G_1$ and $v_0$ are both multivalued, $S$ is no longer multivalued. As $T,S$ are symmetric $(0,2)$ tensors, we could use the Riemannian metric $g_0$ to obtain a $(1,1)$ tensor and $TS$ is the composition of $(1,1)$ tensor, which is also a $(1,1)$ tensor. We could write $T(\na G,v_0)=\Tr(TS)$ and $\lan dG_1,v_0\ran=\Tr(S)$. In a orthonormal frame, we might regard $T$ and $S$ as a $n\ti n$ matrix and $\Tr(TS)$ is understood as trace of the matrix multiplication of $T$ and $S$.
	
	Therefore, we obtain 
	$$
	\frac{d}{ds}|_{s=0}B_s(t)=\int_M(\Tr(TS)-\Tr(T)\Tr(S))dp'.
	$$
	
	\begin{lemma}
		\label{lem_contraction}
		Suppose for any $T$ with $\supp\subset U_l^c$, we have $$\int_M(\Tr(TS)-\Tr(T)\Tr(S))dp'=0,$$ then $S=0$ over $U_l^c$.
	\end{lemma}
	\begin{proof}
		As $T$ is a real valued $(0,2)$ tensor, we only need to show the lemma holds for a real valued $S$. As $g_0$ induces an isomorphism between $T^{\st}M$ and $TM$, we obtain a canonical $(1,1)$ tensor $I$. We define $\mathring{T}=T-\frac{\Tr(T)}{n}I$ and $\mathring{S}=S-\frac{\Tr(S)}{n}I$, which will be the traceless part of $T$ and $S$. Then we have 
		$$\Tr(TS)=\Tr(\mathring{T}\mathring{S})+\frac{1}{n}\Tr(T)\Tr(S).$$
		Therefore, we obtain
		$$
		\int_M(\Tr(\mathring{T}\mathring{S})+\frac{n-1}{n}\Tr(T)\Tr(S)) dp'=0,
		$$
		for any $T$. Let $\chi:U_l^{c}\to [0,1]$ be a cut-off function with $\chi|_{\pa U_l^{c}}=0$, then we first take $T=\chi^2 \mathring{S}$. Then we obtain $\int_M\Tr(\chi\mathring{S}\chi\mathring{S})=0$. As $\Tr(\;)$ is the inner product for symmetric $(0,2)$-tensor, by different choices of $\chi$, we first obtain $\mathring{S}=0$ over $U_l^{c}$.  We take $T=\chi^2\Tr(S)I$, then $\int_M\Tr(S)^2\chi^2dp'=0$, which implies $\Tr(S)=0$ over $U_l^{c}$. In summary, over $U_l^{c}$, we obtain $S=0$.
	\end{proof}
	
	\begin{proposition}
		\label{prop_transversality}
		There exists $T$ with support on $M\setminus \{t\}$ such that $\frac{d}{ds}|_{s=0}\Re{B}_{s}(t)\neq 0$ or $\frac{d}{ds}|_{s=0}\Im{B}_{s}(t)\neq 0$.
	\end{proposition}
	\begin{proof}
		We prove by contradiction. Suppose $\frac{d}{ds}|_{s=0}\Re{B}_{s}(t)$=0, then by Lemma \ref{lem_contraction}, we obtain $\Re{S}(p)=0$ for any $p\in U_l^c$. Let $(x_1,\cdots, x_n)$ be a coordinate on a ball $B_p$ centered at $p$, we write $v_0=\sum_{i=1}^n\al_i dx_i$, then $\Re S(p)=0$ implies $\pa_{x_i}\Re G_1\otimes \al_i=0$ over $B_p$ for $i=1,\cdots,n$. However, by \cite[Equation (12)]{donaldson12kahlermetric}, there is an explicit expression of leading asymptotic of $G_1$ along $L$ such that for $p$ sufficient close to $L$, we have $\pa_{x_i}\Re G_1\neq 0$. Therefore, over $B_p\cap U_l^c$, we obtain $v_0=0$. We choose any $B'\subset B_p\cap U_l^c$ such that $B'\cap L=\emptyset$, then $p^{\st}v_0$ over $p^{\st}B'$ is a smooth harmonic 1-form. As $B'\cap L=\emptyset$, the pull-back metric over $p^{\st}B'$ is also smooth. By the unique continuation theorem of elliptic equation \cite{aronszajn57}, $p^{\st}v_0$ over $p^{-1}(B_p\cap U_l^c)$ could not identically zero, which gives a contradiction. 
		
		As $\pa_{x_i}\Im G_1\neq 0$ on $p$ close to $L$, by the same argument, we could find $T$ such that $\frac{d}{ds}|_{s=0}\Im{B}_{s}(t)\neq 0$.
	\end{proof}
	
	\begin{proposition}
		\label{prop_trans}
		Let $(L,\MI,v)$ be the bounded multivalued harmonic 1-form constructed in Theorem \ref{main_theorem} (i) and $B_v(t)$ be the leading term of $v$, suppose $N$ is a trivial bundle, then for generic $\ZH$ invariant smooth metric, $B_v^{-1}(0)$ is a $n-4$ dimensional submanifold of $L$.
	\end{proposition}
	\begin{proof}
		We define the map $$\MF:L\ti \MM \to \mathbb{C},\;\MF(t,g):=B_g(t).$$ By Proposition \ref{prop_transversality}, $\MF$ is transverse to $\{0\}\subset \mathbb{C}$. Therefore, by Sard-Smale theorem, for generic $g\in \MM$, $B_g^{-1}(0)$ is a n-4 dimensional submanifold.
	\end{proof}
	
	Proof of Theorem \ref{main_theorem} (ii): as every complex line bundle is trivial over a 1-dimensional manifold, the claim follows from Proposition \ref{prop_trans}.
	
	\section{Examples of $\ZH$ symmetric triples}

	In this section, we will construct examples of $\ZH$ symmetric triples using the cyclic branched covering. For the construction of cyclic branched covering and Alexander polynomial, we refer \cite{lickorish1997introducetoknottheory,rolfsen90knotsandlinks} for more details.
	
	\begin{proposition}
		Let $M$ be a rational homology 3-sphere, $L$ be a oriented link on $M$. Let $M_k$ be the $k$-cyclic branched covering of $M$ along $L$. Suppose the 1st Betti number $b_1(M_2)\neq 0$, then there exists $\chi\in H^1(M_6;\mathbb{R})$ such that $(M_6,M_3,\chi)$ is a $\ZH$ symmetric triple.
	\end{proposition}
	\begin{proof}
		By construction of the cyclic branched covering, over $M_6$, there exists an cyclic action $t:M_6\to M_6$ with $t^6=\Id$. We write $\sigma=t^3$ and $\tau=t^2$, then $M_3=M_6/<\sigma>$ and $M_2=M_6/<\tau>$. We write $p_3:M_6\to M_3$ and $p_2:M_6\to M_2$ be the quotient maps, which we make a diagram below to explain their relationship,
		
		\begin{center}
		\begin{tikzpicture}
			\matrix (m) [matrix of math nodes,row sep=3em,column sep=4em,minimum width=2em]
			{
				M_6 & M_3:=M_6/\lan \sigma \ran \\
				\al\in M_2:=M_6/\lan \tau \ran & M=M_6/\lan t\ran \\};
			\path[-stealth]
			(m-1-1) edge node [left] {$p_2$} (m-2-1)
			edge node [above] {$p_3$} (m-1-2)
			(m-2-1.east|-m-2-2) edge node [below] {}
			node [above] {} (m-2-2)
			(m-1-2) edge node [right] {} (m-2-2);
		\end{tikzpicture}
	\end{center}
		
		Let $\al\in H^1(M_2;\mathbb{R})$, we define $\chi=p_2^{\st}\al\in H^1(M_6;\mathbb{R})$. As $p_2\circ\tau=p_2$, $\tau^{\st}\chi=\chi.$ $\sigma$ induces a decomposition $H^1(M_6;\mathbb{R})=H^1_-(M_6;\mathbb{R})\oplus H^1_+(M_6;\mathbb{R})$, which we could write $\chi=\chi^-+\chi^+$ with $\sigma^{\st}\chi^{\pm}=\pm\chi^{\pm}$. As $\tau\sigma=\sigma\tau$, we have $\sigma^{\st}\tau^{\st}\chi^{\pm}=\pm \tau^{\st}\chi^{\pm}$, which implies $\tau^{\st}\chi^+=\chi^+$. 
		
		As $\tau$ and $\sigma$ generates all $\mathbb{Z}_6$ action on $M_6$, we have $t^{\st}\chi^+=\chi^+$. Therefore, $\chi^+$ is the pull-back of an element of $H^1(M;\mathbb{R})$, which vanishes as $M$ is a rational homology 3-sphere. Therefore, $\sigma^{\st}\chi=-\chi$ and $(M_6,M_3,\chi)$ is a $\mathbb{Z}_3$ symmetric triple satisfies Assumption \ref{assumption}.
	\end{proof}
	
	Let $M=S^3$ and $L$ be an oriented link on $M$ with $n$ components, then the Alexander polynomial is an oriented link invariant, which could be written as
	$$\Delta_L(t)=\det(tS-S^{\perp}),$$
	where $S$ is the Seifert matrix for $L$. In addition, when $t=-1$, $\det(L)=\Delta_L(-1)$ is called the determinant of $L$.
	
	Let $S_k(L)$ be the k-cyclic branched covering of $L$, then Alexander polynomial characterize the order of the 1st homology of the cyclic branched cover.
	\begin{theorem}{\cite{lickorish1997introducetoknottheory}}
		The order of the first homology group of $S_k^3(L)$ is given by
		$$
		|H_1(S_k(L))|=|\prod_{r=1}^{k-1}\Delta_L(e^{\frac{2\pi i r}{k}})|,
		$$
		where $|.|$ means the order a group which is defined to the zero if the group has infinite order.
	\end{theorem}
	In particular, when $k=2$, $|H_1(S_2(L))|=|\Delta_L(-1)|=|\det(L)|$. In particular, $b_1(S_2(L))>0$ if and only if $\det(L)=0$. As the determinant for a knot is also always not zero, therefore, a multivalued harmonic 1-form could only branched over a link with more than one components, which is first observed by Haydys \cite{haydys2020seiberg}.
	
	In summary, we obtain the following existence result.
	\begin{corollary}
		\label{cor}
		Let $L$ be an oriented link on $S^3$ with $\det(L)=0$, then over the 3-cyclic branched covering $S_3(L)$, there exists a bounded multivalued harmonic 1-form.
	\end{corollary}
	
	Using LinkInfo \cite{linkinfo}, we find all determinant 0 links with crossing less than 11 such that the three cyclic branched covering is a rational homology sphere, which we will list in the Appendix \ref{linktablelist}.
	
	Let $L_1,L_2$ be two oriented links on $S^3$ and let $L:=L_1+L_2$ be the disjoint union links, then $\Delta_L(t)=\Delta_{L_1}(t)\Delta_{L_1}(t)$. Let $\zeta=e^{\frac{2\pi i}{3}}$, suppose $\Delta_{L_1}(-1)=0$ and $\Delta_{L_i}(\zeta)\Delta_{L_i}(\zeta^2)\neq 0$, then $b_1(S_2(L))>0$, $S_3(L)$ is a rational homology sphere. Using disjoint union, we could construct infinity number of rational homology spheres satisfies the condition of Corollary \ref{cor}.
	
	Proof of Theorem \ref{thm_infinity}: Let $L$ be any oriented links in the list in Appendix \ref{linktablelist} and $K$ be a oriented Trefoil with Alexander polynomial $\Delta_{K}(t)=1-t+t^2$. Let $\zeta=e^{\frac{2\pi i}{3}}$, then $|\Delta_{K}(\zeta)\Delta_{K}(\zeta^2)|=4$. We consider the link $L_n:=L\cup n K$, which is the disjoint union of $n$-copies of $K$. Then $\Delta_{L_n}(-1)=0$ and $|\Delta_{L_n}(\zeta)\Delta_{L_n}(\zeta^2)|=4^n|\Delta_L(\zeta)\Delta_L(\zeta^2)|$. Let $S_3(L_n)$ be the 3-cyclic branched covering of $L_n$, then $$|H_1(S_3(L_n))|=4^n|\Delta_L(\zeta)\Delta_L(\zeta^2)|\neq 0,$$ which are rational homology spheres. In addition, as the order of $H_1(S_3(L_n))$ are different, $S_3(L_n)$ will be different 3-manifolds. Therefore, we have construct infinity number of rational homology 3-sphere that admits bounded multivalued harmonic 1-form. \qed

	\newpage
	\appendix
	
	\section{}
	\label{linktablelist}
	The following table is a list of oriented links $L$ with crossings less than 11 and the 3-cyclic branched covering is a rational homology sphere. We use the date on Linkinfo \cite{linkinfo}. In particular, over $S_3(L)$, there exists a bounded multivalued harmonic 1-form. The first column contains the name of the oriented link $L$, the second column has the normalized Alexander polynomial, and the third column contains the order of $S_3(L)$'s first homology.

	\begin{tabular}{ |p{3cm}||p{6cm}|p{3cm}|  }
		\hline
		Name of Link $L$ & $\frac{\Delta_L}{t-1}$ &$|H_1(S_3(L))|$\\
		\hline
		L8n6\{0;0\}	&$	-t-t^2+t^3+t^4	$&	9	\\
		L8n6\{1;0\}	&$	-1+t^2-t^3+t^5	$&	36	\\
		L8n6\{0;1\}	&$	t+t^2-t^3-t^4	$&	9	\\
		L8n6\{1;1\}	&$	t+t^2-t^3-t^4	$&	9	\\
		L8n8\{1;0;1\}	&$	1-2t^2+t^4	$&	27	\\
		L8n8\{0;1;1\}	&$	-1+2t^2-t^4	$&	27	\\
		L9n18\{0\}	&$	-1-2t^3-t^6	$&	48	\\
		L9n18\{1\}	&$	2t^3+t^2+t^4	$&	3	\\
		L9n19\{0\}	&$	-t-t^2-t^4-t^5	$&	12	\\
		L9n19\{1\}	&$	t+t^4+t^2+t^5	$&	12	\\
		L10n56\{0\}	&$	1-2t^2+t^4	$&	27	\\
		L10n56\{1\}	&$	1-2t^2+t^4	$&	27	\\
		L10n57\{0\}&$	1-2t^2+t^4	$&	27	\\
		L10n57\{1\}	&$	1-2t^2+t^4	$&	27	\\
		L10n91\{0;0\}	&$	1-3t^2-2t+3t^3+2t^4-t^5	$&	144	\\
		L10n91\{1;0\}	&$	t^2+t^3-t^4-t^5	$&	9	\\
		L10n91\{0;1\}&$	t+t^2-t^3-t^4	$&	9	\\
		L10n91\{1;1\}	&$	t+t^2-t^3-t^4	$&	9	\\
		L10n93\{0;0\}	&$	-1-t^3+t^4+t^7	$&	36	\\
		L10n93\{1;0\}&$	t^2+t^3-t^4-t^5	$&	9	\\
		L10n93\{1;1\}	&$	t^2+t^3-t^4-t^5	$&	9	\\
		L10n94\{1;0\}	&$	t^2+t^3-t^4-t^5	$&	9	\\
		L10n94\{0;1\}	&$	-t^2-t+t^6+t^5	$&	9	\\
		L10n94\{1;1\}	&$	t^2+t^3-t^4-t^5	$&	9	\\
		L10n104\{0;1;0\}	&$	-t+2t^3-t^5	$&	27	\\
		L10n104\{1;0;1\}	&$	1-t^2-t^4+t^6	$&	27	\\
		L10n104\{1;1;1\}	&$	-t+2t^3-t^5	$&	27	\\
		L10n111\{0;0;0\}	&$	t-2t^3+t^5	$&	27	\\
		L10n111\{1;1;0\}	&$	t-2t^3+t^5	$&	27	\\
		L10n111\{0;0;1\}	&$	2t+4t^3-2t^5	$&	84	\\
		L10n111\{1;1;1\}	&$	-1+2t+t^2-4t^3+t^4+2t^5-t^6	$&	243\\
		\hline
	\end{tabular}
	
	\bibliographystyle{plain}
	\bibliography{references}
\end{document}